\documentclass{amsart}

\usepackage{amsmath}
\usepackage{amsthm}
\usepackage{amsfonts}
\usepackage{dsfont}
\usepackage{color}
\usepackage{enumerate}

\newcommand{\reals}{\mathbb{R}}

\newcommand{\complex}{\mathbb{C}}
\newcommand{\naturals}{\mathbb{N}}

\newcommand{\bracketb}[1]{\Big[#1\Big]}
\newcommand{\bracketc}[1]{\bigg[#1\bigg]}

\newcommand{\paraa}[1]{\big(#1\big)}
\newcommand{\parab}[1]{\Big(#1\Big)}

\newcommand{\diag}{\operatorname{diag}}

\newcommand{\im}{\operatorname{im}}

\newcommand{\spacearound}[1]{\quad#1\quad}

\renewcommand{\implies}{\spacearound{\Rightarrow}}

\newtheorem{theorem}{Theorem}[section]
\newtheorem{corollary}[theorem]{Corollary}
\newtheorem{lemma}[theorem]{Lemma}
\newtheorem{proposition}[theorem]{Proposition}
\newtheorem{example}[theorem]{Example}
\theoremstyle{definition}
\newtheorem{definition}[theorem]{Definition}
\theoremstyle{remark}
\newtheorem{remark}[theorem]{Remark}
\numberwithin{equation}{section}

\newcommand{\tr}{\operatorname{tr}}

\newcommand{\A}{\mathcal{A}}

\newcommand{\Kfield}{\mathbb{K}}
\renewcommand{\P}{\mathcal{P}}
\newcommand{\Ph}{\hat{\P}}

\newcommand{\R}{\mathcal{R}}
\newcommand{\Rh}{\widehat{\R}}

\newcommand{\gb}{\,\bar{\!g}}

\renewcommand{\d}{\partial}

\newcommand{\eps}{\varepsilon}

\newcommand{\nablab}{\bar{\nabla}}

\renewcommand{\mid}{\mathds{1}}
 
\newcommand{\D}{\mathcal{D}}

\newcommand{\gi}{\gamma^{-1}}
\newcommand{\gover}{\frac{1}{\gamma}}

\newcommand{\KP}{K\"ahler--Poisson}

\newcommand{\Der}{\operatorname{Der}}
\newcommand{\X}{\mathcal{X}}

\renewcommand{\div}{\operatorname{div}}

\newcommand{\Ric}{\operatorname{Ric}}

\newcommand{\K}{\mathcal{K}}

\newcommand{\TA}{T\mathcal{A}}

\newcommand{\gh}{\hat{g}}

\newcommand{\g}{\mathfrak{g}}
\newcommand{\half}{\frac{1}{2}}

\newcommand{\End}{\operatorname{End}}

\newcommand{\mathand}{\quad\text{ and }\quad}

\newcommand{\Vt}{\tilde{V}}
\newcommand{\Afin}{\A_{\textrm{fin}}}
\newcommand{\Phind}[2]{{\Ph^{#1}}{}_{#2}}

\title{K\"ahler--Poisson algebras}

\author{Joakim Arnlind}
\address[Joakim Arnlind]{Dept. of Math.\\
Link\"oping University\\
581 83 Link\"oping\\
Sweden}
\email{joakim.arnlind@liu.se}

\author{Ahmed Al-Shujary}
\address[Ahmed Al-Shujary]{Dept. of Math.\\
Link\"oping University\\
581 83 Link\"oping\\
Sweden}
\email{ahmed.al-shujary@liu.se}

\thanks{}

\subjclass[2000]{}
\keywords{}

\begin{document}

\begin{abstract}
  We introduce \KP\ algebras as analogues of algebras of smooth
  functions on K\"ahler manifolds, and prove that they share
  several properties with their classical counterparts on an algebraic
  level. For instance, the module of inner derivations of a \KP\
  algebra is a finitely generated projective module, and allows for a
  unique metric and torsion-free connection whose curvature enjoys all
  the classical symmetries. Moreover, starting from a large class of
  Poisson algebras, we show that every algebra has an associated \KP\
  algebra constructed as a localization. At the end, detailed
  examples are provided in order to illustrate the novel concepts.
\end{abstract}

\maketitle

\section{Introduction}\label{sec:introduction}

\noindent
Poisson manifolds and their geometry have been of great interest over
the last decades. Besides from being important from a purely
mathematical point of view, they are also fundamental to areas in
mathematical and theoretical physics. Many authors have studied the
geometric and algebraic properties of symplectic and Poisson manifolds
together in relation to concepts such as connections, local structure
and cohomology (see
e.g. \cite{l:poisson.lie,a:local.structure.poisson,b:differential.complex.poisson,h:poisson.cohomology}). Moreover,
there is a well developed field of deformations of Poisson structures,
perhaps most famous through Kontsevich's result on the existence of
formal deformations \cite{k:deformation.quantization}.  The ring of
smooth functions on a Poisson manifold is a Poisson algebra and it
seems quite natural to ask to what extent geometric properties and
concepts may be introduced in an arbitrary Poisson algebra, without
making reference to an underlying manifold. The methods of algebraic
geometry can readily be extended to Poisson algebras (see
e.g. \cite{b:poisson.algebraic.geometry}); however, this will not be
directly relevant to us as we shall start by focusing on metric
aspects.  Our work is mainly motivated by the results in
\cite{ahh:multi.linear,ah:pseudo.riemannian}, where it is shown that
one may reformulate the Riemannian geometry of an embedded K\"ahler
manifold $M$ entirely in terms of the Poisson structure on the algebra
smooth functions of $M$. Let us also mention that the starting point
of our approach is quite similar to that of
\cite{h:poisson.cohomology} (although metric aspects were not
considered there).

In this note, we show that any Poisson algebra, fulfilling an ``almost
K\"ahler condition'', enjoys many properties similar to those of the
algebra of smooth functions on an almost K\"ahler manifold, opening up
for a more metric treatment of Poisson algebras. Such algebras will be
called ``\KP\ algebras'', and we show that one may associate a \KP\
algebra to every algebra in a large class of Poisson algebras.  In
particular, we prove the existence of a unique Levi-Civita connection
on the module generated by the inner derivations, and show that the
curvature operator has all the classical symmetries. As our approach
is quite close to the theory of Lie-Rinehart algebras, we start by
introducing metric Lie-Rinehart algebras and recall a few results on
the Levi-Civita connection and the corresponding curvature.

In physics, the dynamics of quantum systems are found by using a
correspondence between Poisson brackets of functions on the classical
manifold, and the commutator of operators in the quantum system. Thus,
understanding how properties of the underlying manifold may be
expressed in Poisson algebraic terms enables both interpretation and
definition of quantum mechanical quantities. For instance, this has
been used in the context of matrix models to identify emergent
geometry
(cf. \cite{bs:curvature.gravity.matrix.models,ahh:multi.linear}).

Let us briefly outline the contents of the paper. In
Section~\ref{sec:poisson.algebraic.geometry} we recall a few the
results from \cite{ah:pseudo.riemannian}, in order to motivate and
understand the introduction of a K\"ahler type condition for Poisson
algebras, and Section~\ref{sec:metric.Lie.Rinehart} explains how the
theory of Lie-Rinehart algebras can be extended to include metric
aspects. In Section \ref{sec:kpalgebras}, we define \KP\ algebras and
investigate their basic properties as well as showing that one may
associate a \KP\ algebra to an arbitrary Poisson algebra in a large
class of algebras.  In Section~\ref{sec:levi.civita} we derive a
compact formula for the Levi-Civita connection as well as introducing
Ricci and scalar curvature. Section~\ref{sec:examples} presents a
number of examples together with a few detailed computations.

\begin{remark}
  We have become aware of the fact that the terminology \emph{\KP\
    structure} (resp. \emph{\KP\ manifold}) is used for certain
  Poisson structures on a complex manifold where the Poisson bivector
  is of type $(1,1)$ (see
  e.g. \cite{k:covariant.quantization.separation}), but we hope that
  this will not be a source of confusion for the reader.
\end{remark}

\section{Poisson algebraic formulation of almost K\"ahler manifolds}
\label{sec:poisson.algebraic.geometry}

\noindent In \cite{ah:pseudo.riemannian} it was shown that the
geometry of embedded almost K\"ahler manifolds can be reformulated
entirely in the Poisson algebra of smooth functions.  As we shall
develop an algebraic analogue of this fact, let us briefly recall the
main construction.

Let $(\Sigma,\omega)$ denote a $n$-dimensional symplectic manifold and
let $g$ be a metric on $\Sigma$. Furthermore, let us assume that
$x:(\Sigma,g)\to(\reals^m,\gb)$ is an isometric embedding of $\Sigma$
into $\reals^m$ (with the metric $\gb$), and write
\begin{align*}
  p\to x(p)=\paraa{x^1(p),x^2(p),\ldots,x^m(p)}.
\end{align*}
The results in \cite{ah:pseudo.riemannian} state that the Riemannian
geometry of $\Sigma$ may be formulated in terms of the Poisson algebra
generated by the embedding coordinates $x^1,\ldots,x^m$. These results
hold true as long as there exists a non-zero function
$\gamma\in C^\infty(\Sigma)$ such that
\begin{align}\label{eq:gthetarel}
  \gamma^2g^{ab} = \theta^{ap}\theta^{bq}g_{pq}
\end{align}
where $\theta^{ab}$ and $g_{ab}$ denote the components of the Poisson
bivector and the metric in local coordinates $\{u^a\}_{a=1}^n$,
respectively. If $(\Sigma,\omega,g)$ is an almost K\"ahler manifold
then it follows from the compatibility condition
$\omega(X,Y)=g(X,J(Y))$ (where $J$ denotes the almost complex
structure on $\Sigma$) that relation \eqref{eq:gthetarel} holds with
$\gamma=1$. In local coordinates, the isometric embedding is
characterized by
\begin{align*}
  g_{ab} = \gb_{ij}\paraa{\d_ax^i}\paraa{\d_bx^j},
\end{align*}
and the Poisson bracket is computed as
\begin{align*}
  \{f,h\} = \theta^{ab}(\d_af)(\d_bh).
\end{align*}
Note that in the above and following formulas, indices $i,j,k,\ldots$
run from $1$ to $m$ and indices $a,b,c,\ldots$ run from $1$ to $n$.

Defining $\D:T_p\reals^m\to T_p\reals^m$ as
\begin{align*}
  \D(X) \equiv {\D^i}_jX^j\d_i=\frac{1}{\gamma^2}\{x^i,x^k\}\gb_{kl}\{x^j,x^l\}\gb_{jm}X^m\d_i
\end{align*}
for $X=X^i\d_i\in T_p\reals^m$, one computes
\begin{align*}
  \D(X)^i &= \frac{1}{\gamma^2}\theta^{ab}(\d_ax^i)(\d_bx^k)\gb_{kl}
            \theta^{pq}(\d_px^j)(\d_qx^l)\gb_{jm}X^m\\
          &=\frac{1}{\gamma^2}\theta^{ab}\theta^{pq}g_{bq}(\d_ax^i)(\d_px^j)\gb_{jm}X^m
            =g^{ap}(\d_ax^i)(\d_px^j)\gb_{jm}X^m,
\end{align*}
by using \eqref{eq:gthetarel}.  Hence, the map $\D$ is identified as
the orthogonal projection onto $T_p\Sigma$, seen as a subspace of
$T_p\reals^m$. Having the projection operator at hand, one may
directly proceed to develop the theory of submanifolds. For instance,
the Levi-Civita connection $\nabla$ on $\Sigma$ is given by
\begin{align*}
  \nabla_XY = \D\paraa{\nablab_XY}
\end{align*}
where $X,Y\in\Gamma(T\Sigma)$ and $\nablab$ is the Levi-Civita
connection on $(\reals^m,\gb)$. In the particular case (but
generically applicable, by Nash's theorem \cite{nash:imbedding}) when
$\gb$ is the Euclidean metric, the above formula reduces to
\begin{align*}
  \nabla_XY^i = \frac{1}{\gamma^4}\sum_{i,j,k,l,n=1}^m
  \{x^i,x^k\}\{x^j,x^k\}X^l\{x^l,x^n\}\{Y^j,x^n\}.
\end{align*}

\noindent As we intend to develop an analogous theory for Poisson
algebras, without any reference to a manifold, we would like to
reformulate \eqref{eq:gthetarel} in terms of Poisson algebraic
expressions. Using that $g_{ab}=\gb_{ij}(\d_ax^i)(\d_bx^j)$ and
$\{x^i,x^k\}=\theta^{ab}(\d_ax^i)(\d_bx^j)$, one derives
\begin{align*}
  &\gamma^2g^{ab}=\theta^{ap}\theta^{bq}g_{pq}\implies
  \gamma^2\delta^a_c=\theta^{ap}\theta^{bq}g_{pq}g_{bc}\implies
  \gamma^2\theta^{ar} =
  \theta^{ap}\theta^{bq}g_{pq}g_{bc}\theta^{cr}\\
  &\implies
  \gamma^2\{x^i,x^j\} =
  (\d_ax^i)(\d_rx^j)\theta^{ap}\theta^{bq}\theta^{cr}
  \gb_{kl}(\d_px^k)(\d_qx^l)\gb_{mn}(\d_bx^m)(\d_cx^n)\\
  &\implies
  \gamma^2\{x^i,x^j\} = -\{x^i,x^k\}\gb_{kl}\{x^l,x^n\}\gb_{nm}\{x^m,x^j\}
\end{align*}
which is equivalent to the statement that
\begin{align}
  \label{eq:kpcond.f.h}
  \gamma^2\{f,h\} = -\{f,x^i\}\gb_{ij}\{x^j,x^k\}\gb_{kl}\{x^l,h\}
\end{align}
for all $f,h\in C^\infty(\Sigma)$. Given $\gamma^2$, $\gb_{ij}$ and
$x^1,\ldots,x^m$, the above equation makes sense in an arbitrary
Poisson algebra. The main purpose of this paper is to study algebras
which satisfy such a relation.

\section{Metric Lie-Rinehart algebras}\label{sec:metric.Lie.Rinehart}

\noindent
The idea of modeling the algebraic structures of differential geometry
in a commutative algebra is quite old. We shall follow a pedestrian
approach, were we assume that a (commutative) algebra $\A$ is given
(corresponding to the algebra of functions), together with an
$\A$-module $\g$ (corresponding to the module of vector fields) which
is also a Lie algebra and has an action on $\A$ as derivations. Under
appropriate assumptions on the ingoing objects, such systems has been
studied by many authors over the years, see e.g
\cite{h:pseudo.algebre.lie,k:fibre.bundles,p:cohomology.lie.rings,%
  r:differential.forms,n:tensoranalysis,h:poisson.cohomology}. Our
starting point is the definition given by G. Rinehart
\cite{r:differential.forms}. In the following, we let the field $\Kfield$
denote either $\reals$ or $\complex$.

\begin{definition}[Lie-Rinehart algebra]\label{def:Lie.Rinehart.algebra}
  Let $\A$ be a commutative $\Kfield$-algebra and let $\g$ be an
  $\A$-module which is also a Lie algebra over $\Kfield$. Given a map
  $\omega:\g\to\Der(\A)$, the pair $(\A,\g)$ is called a
  \emph{Lie-Rinehart} algebra if
  \begin{align}
    &\omega(a\alpha)(b) = a\paraa{\omega(\alpha)(b)}\\
    &[\alpha,a\beta] = a[\alpha,\beta]+\paraa{\omega(\alpha)(a)}\beta,
  \end{align}
  for $\alpha,\beta\in\g$ and $a,b\in\A$. (In most cases, we will
  leave out $\omega$ and write $\alpha(a)$ instead of
  $\omega(\alpha)(a)$.)
\end{definition}

\noindent
Let us point out some immediate examples of Lie-Rinehart algebras.

\begin{example}
  Let $\A$ be an algebra and let $\g=\Der(\A)$ be the $\A$-module
  of derivations of $\A$. It is easy to check that $\Der(\A)$ is a
  Lie algebra with respect to composition of derivations, i.e.
  \begin{align*}
    [\alpha,\beta](a) = \alpha(\beta(a))-\beta(\alpha(a)).
  \end{align*}
  The pair $(\A,\Der(\A))$ is a Lie-Rinehart algebra with respect to
  the action of elements of $\Der(\A)$ as derivations.
\end{example}

\begin{example}
  Let $\A=C^\infty(M)$ be the algebra (over $\reals$) of smooth
  functions on a manifold $M$, and let $\g=\X(\A)$ be the $\A$-module of
  vector fields on $M$. With respect to the standard action of a
  vector field as a derivation of $C^\infty(M)$, the pair
  $(C^\infty(M),\X(\A))$ is a Lie-Rinehart algebra. 
\end{example}

\noindent 
Morphisms of Lie-Rinehart algebras are defined as follows.

\begin{definition}
  Let $(\A_1,\g_1)$ and $(\A_2,\g_2)$ be Lie-Rinehart algebras. A
  \emph{morphism} of Lie-Rinehart algebras is a pair of maps
  $(\phi,\psi)$, with $\phi:\A_1\to \A_2$ an algebra homomorphism
  and $\psi:\g_1\to\g_2$ a Lie algebra homomorphism, such
  that 
  \begin{align*}
    \psi(a\alpha) = \phi(a)\psi(\alpha)\quad\text{and}\quad
    \phi\paraa{\alpha(a)} = \psi(\alpha)\paraa{\phi(a)},
  \end{align*}
  for all $a\in\A_1$ and $\alpha\in\g_1$.
\end{definition}

\noindent
A lot of attention has been given to the cohomology of the
Chevalley--Eilenberg complex consisting of alternating
$\A$-multilinear maps with values in a module $M$. Namely, defining
$C^k(\g,M)$ to be the $\A$-module of alternating maps from $\g^k$ to
an $(\A,\g)$-module $M$, on introduces the standard differential
$d:C^k(\g,M)\to C^{k+1}(\g,M)$ as
\begin{equation}
  \label{eq:CE.differential}
  \begin{split}
    d\tau(\alpha_1,\ldots,\alpha_{k+1}) &= 
    \sum_{i=1}^{k+1}(-1)^{i+1}\alpha_i\paraa{\tau(\alpha_1,\ldots,\hat{\alpha}_i,\ldots,\alpha_{k+1})}\\
    &\quad +
  \sum_{i<j}^{k+1}(-1)^{i+j}\tau\paraa{[\alpha_i,\alpha_j],\alpha_1,
    \ldots,\hat{\alpha}_i,\ldots,\hat{\alpha}_j,\ldots,\alpha_{k+1}},
  \end{split}
\end{equation}
where $\hat{\alpha}_i$ indicates that $\alpha_i$ is \emph{not} present
among the arguments. The fact that $d\circ d=0$ implies that one can
construct the cohomology of this complex in analogy with de Rahm
cohomology of smooth manifolds.  However, as we shall be more
interested in Riemannian aspects, it is natural to study the case when
there exists a metric on the module $\g$. More precisely, we make the
following definition.

\begin{definition}\label{def:metric}
  Let $(\A,\g)$ be a Lie-Rinehart algebra and let $M$ be an
  $\A$-module. An $\A$-bilinear form $g:M\times M\to\A$ is
  called a \emph{metric on $M$} if it holds that
  \begin{enumerate}
  \item $g(m_1,m_2)=g(m_2,m_1)$ for all $m_1,m_2\in M$,
  \item the map $\gh:M\to M^\ast$, given by
    $\paraa{\gh(m_1)}(m_2)=g(m_1,m_2)$, is an $\A$-module isomorphism,\label{def.metric.iso}
  \end{enumerate}
  where $M^\ast$ denotes the dual of $M$. We shall often refer to
  property (\ref{def.metric.iso}) as the metric being \emph{non-degenerate}.
\end{definition}

\begin{definition}
  A \emph{metric Lie-Rinehart algebra $(\A,\g,g)$} is a Lie-Rinehart
  algebra $(\A,\g)$ together with a metric $g:\g\times\g\to\A$.
\end{definition}

\noindent
Let us introduce morphisms of metric Lie-Rinehart algebras as
morphisms of Lie-Rinehart algebras that preserve the metric.

\begin{definition}
  Let $(\A_1,\g_1,g_1)$ and $(\A_2,\g_2,g_2)$ be metric Lie-Rinehart
  algebras. A \emph{morphism of metric Lie-Rinehart algebras} is a
  morphism of Lie-Rinehart algebras
  $(\phi,\psi):(\A_1,\g_1)\to(\A_2,\g_2)$ such that
  \begin{align*}
    \phi\paraa{g_1(\alpha,\beta)} = g_2\paraa{\psi(\alpha),\psi(\beta)}
  \end{align*}
  for all $\alpha,\beta\in\g_1$.
\end{definition}

\noindent
The theory of affine connections can readily be introduced, together
with torsion-freeness and metric compatibility.

\begin{definition}\label{def:Lie.Rinehart.connection}
  Let $(\A,\g)$ be a Lie-Rinehart algebra and let $M$ be an
  $\A$-module. A \emph{connection $\nabla$ on $M$} is a map
  $\nabla:\g\to\End_{\Kfield}(M)$, written as
  $\alpha\to\nabla_\alpha$, such that
  \begin{enumerate}
  \item $\nabla_{a\alpha+\beta}=a\nabla_\alpha+\nabla_\beta$
  \item $\nabla_\alpha(am) = a\nabla_\alpha m+\alpha(a)m$
  \end{enumerate}
  for all $a\in\A$, $\alpha,\beta\in\g$ and $m\in M$.
\end{definition}

\begin{definition}
  Let $(\A,\g)$ be a Lie-Rinehart algebra and let $M$ be an
  $\A$-module with connection $\nabla$ and metric $g$. The connection is called
  \emph{metric} if 
  \begin{align}
    \alpha\paraa{g(m_1,m_2)} = g(\nabla_\alpha
    m_1,m_2)+g(m_1,\nabla_\alpha m_2)
  \end{align}
  for all $\alpha\in\g$ and $m_1,m_2\in M$.
\end{definition}

\begin{definition}
  Let $(\A,\g)$ be a Lie-Rinehart algebra and let $\nabla$ be a
  connection on $\g$. The connection is called \emph{torsion-free} if
  \begin{align*}
    \nabla_\alpha\beta-\nabla_\beta\alpha-[\alpha,\beta] = 0
  \end{align*}
  for all $\alpha,\beta\in\g$.
\end{definition}

\noindent
As in differential geometry, one can show that there exists a unique
torsion-free and metric connection associated to the Riemannian
metric. The first step involves proving Kozul's formula.

\begin{proposition}\label{prop:Kozul.formula}
  Let $(\A,\g,g)$ be a metric Lie-Rinehart algebra. If $\nabla$ is a
  metric and torsion-free connection on $\g$ then it holds that
  \begin{equation}
    \begin{split}\label{eq:Kozul.formula}
      2g\paraa{\nabla_\alpha\beta,\gamma} = 
      \alpha\paraa{&g(\beta,\gamma)}+\beta\paraa{g(\gamma,\alpha)}-\gamma\paraa{g(\alpha,\beta)}\\
      &+g(\beta,[\gamma,\alpha])+g(\gamma,[\alpha,\beta])-g(\alpha,[\beta,\gamma])
    \end{split}
  \end{equation}  
  for all $\alpha,\beta,\gamma\in\g$.
\end{proposition}

\begin{proof}
  Starting from the right-hand-side of \eqref{eq:Kozul.formula} and using
  the metric condition to rewrite the first three terms as
  \begin{align*}
    \alpha\paraa{g(\beta,\gamma)} = g(\nabla_\alpha \beta,\gamma) + g(\beta,\nabla_\alpha\gamma),
  \end{align*}
  together with the torsion-free condition to rewrite the last three
  terms as
  \begin{align*}
    g(\beta,[\gamma,\alpha]) = g(\beta,\nabla_\gamma\alpha)-g(\beta,\nabla_\alpha\gamma),
  \end{align*}
  immediately gives $2g(\nabla_\alpha \beta,\gamma)$.
\end{proof}

\noindent
By using Proposition~\ref{prop:Kozul.formula} together with the fact
that the metric is non-degenerate, one obtains the following result.

\begin{proposition}\label{prop:levi.civita.thm}
  Let $(\A,\g,g)$ be a metric Lie-Rinehart algebra. Then there exists a
  unique metric and torsion-free connection on $\g$.
\end{proposition}

\begin{remark}
  The unique connection in Proposition~\ref{prop:levi.civita.thm}
  will be referred to as the \emph{Levi-Civita connection} of a metric
  Lie-Rinehart algebra. 
\end{remark}

\begin{proof}
  For every $\alpha,\beta\in\g$, the right-hand-side of
  \eqref{eq:Kozul.formula} defines a linear form $\omega\in\g^\ast$
  \begin{align*}
      \omega(\gamma) = 
      \tfrac{1}{2}\alpha\paraa{&g(\beta,\gamma)}+\tfrac{1}{2}\beta\paraa{g(\gamma,\alpha)}-\tfrac{1}{2}\gamma\paraa{g(\alpha,\beta)}\\
      &+\tfrac{1}{2}g(\beta,[\gamma,\alpha])+\tfrac{1}{2}g(\gamma,[\alpha,\beta])-\tfrac{1}{2}g(\alpha,[\beta,\gamma]).
  \end{align*}
  By assumption (see Definition~\ref{def:metric}), the metric induces
  an isomorphism map $\hat{g}:\g\to\g^\ast$, which implies that there
  exists an element $\nabla_\alpha \beta=\gh^{-1}(\omega)\in\g$ such
  that $g(\nabla_\alpha \beta,\gamma)=\omega(\gamma)$. This shows that
  $\nabla_\alpha \beta$ exists for all $\alpha,\beta\in\g$ such that
  relation \eqref{eq:Kozul.formula} is satisfied. Next, let us show
  that $\nabla$ defines a connection on $\g$, which amounts to
  checking the four properties in
  Definition~\ref{def:Lie.Rinehart.connection}. This is a straight-forward
  computation using \eqref{eq:Kozul.formula} and the fact that, for
  instance,
  \begin{align*}
    g(\nabla_{a\alpha}\beta,\gamma)=g(a\nabla_\alpha \beta,\gamma)\qquad\text{for all } \gamma\in\g
  \end{align*}
  implies that $\nabla_{a\alpha}\beta=a\nabla_\alpha \beta$ since the
  metric is non-degenerate. Let us illustrate the computation with the
  following example. From \eqref{eq:Kozul.formula} it follows that
  \begin{align*}
    2g(&\nabla_{a\alpha}\beta,\gamma) = a\alpha\paraa{g(\beta,\gamma)}+\beta\paraa{g(\gamma,a\alpha)}-\gamma\paraa{g(a\alpha,\beta)}\\
      &\qquad\qquad\qquad+g(\beta,[\gamma,a\alpha])+g(\gamma,[a\alpha,\beta])-g(a\alpha,[\beta,\gamma])\\
      &= a\alpha\paraa{g(\beta,\gamma)} + a\beta\paraa{g(\gamma,\alpha)} + \beta(a)g(\gamma,\alpha)
      -a\gamma\paraa{g(\alpha,\beta)}-\gamma(a)g(\alpha,\beta)\\
      &\quad
      +g(\beta,\gamma(a)\alpha+a[\gamma,\alpha])
      +g(\gamma,-\beta(a)\alpha+a[\alpha,\beta])
      -ag(\alpha,[\beta,\gamma])\\
      &=2ag(\nabla_\alpha \beta,\gamma) + \beta(a)g(\gamma,\alpha) - \gamma(a)g(\alpha,\beta)
      +\gamma(a)g(\beta,\alpha)-\beta(a)g(\gamma,\alpha)\\
      &=2ag(\nabla_\alpha \beta,\gamma).
  \end{align*}
  The remaining properties of a connection is proved in an analogous
  way. To show that $\nabla$ is metric, one again uses
  \eqref{eq:Kozul.formula} to substitute $g(\nabla_\alpha
  \beta,\gamma)$ and $g(\beta,\nabla_\alpha \gamma)$ and find that
  \begin{align*}
    \alpha\paraa{g(\beta,\gamma)}-g(\nabla_\alpha \beta,\gamma)-g(\beta,\nabla_\alpha \gamma) = 0.
  \end{align*}
  That the torsion-free condition holds follows from
  \begin{align*}
    g(\nabla_\alpha \beta,\gamma)-g(\nabla_\beta\alpha,\gamma)-g([\alpha,\beta],\gamma)=0,
  \end{align*}
  which can be seen using \eqref{eq:Kozul.formula}. Hence, we conclude
  that there exists a metric and torsion-free affine connection
  satisfying \eqref{eq:Kozul.formula}. Moreover, since the metric is
  non-degenerate, such a connection is unique. Finally, as every
  metric and torsion-free connection on $\g$ satisfies
  \eqref{eq:Kozul.formula} (by Proposition~\ref{prop:Kozul.formula})
  we conclude that there exists a unique metric and torsion-free
  connection on $\g$.
\end{proof}

\noindent
In what follows, we shall recall some of the properties satisfied by a
metric and torsion-free connection. The differential geometric proofs
goes through with only a change in notation needed, but we provide
them here for easy reference, and to adapt the formulation to our
particular situation. We refer to
\cite{k:fibre.bundles,n:tensoranalysis} for a nice overview of
differential geometric constructions in modules over general commutative
algebras.

Following the usual definitions, we introduce the
curvature as
\begin{align}
  R(\alpha,\beta)\gamma = \nabla_\alpha\nabla_\beta \gamma-\nabla_\beta\nabla_\alpha \gamma-\nabla_{[\alpha ,\beta]}\gamma
\end{align}
as well as 
\begin{align*}
  &R(\alpha,\beta,\gamma) = R(\alpha,\beta)\gamma\\
  &R(\alpha,\beta,\gamma,\delta) = g(\alpha,R(\gamma,\delta)\beta).
\end{align*}
Let us also consider the extension of $\nabla$ to
multilinear maps $T:\g^k\to\A$ 
\begin{align*}
  (\nabla_\beta T)(\alpha_1,\ldots,\alpha_k)
  =\beta\paraa{T(\alpha_1,\ldots,\alpha_k)}
  -\sum_{i=1}^kT\paraa{\alpha_1,\ldots,\nabla_\beta \alpha_i,\ldots,\alpha_k},
\end{align*}
as well as to $\g$-valued multilinear maps $T:\g^k\to\g$
\begin{align*}
  (\nabla_\beta T)(\alpha_1,\ldots,\alpha_k)
  =\nabla_\beta\paraa{T(\alpha_1,\ldots,\alpha_k)}
  -\sum_{i=1}^kT\paraa{\alpha_1,\ldots,\nabla_\beta \alpha_i,\ldots,\alpha_k}.
\end{align*}

\noindent
As in classical geometry, one proceeds to derive the Bianchi
identities.

\begin{proposition}\label{prop:bianchiIdentities}
  Let $\nabla$ be the Levi-Civita connection of a metric Lie-Rinehart
  algbera $(\A,\g,g)$ and let $R$ denote corresponding curvature. Then
  it holds that
  \begin{align}
    &R(\alpha,\beta,\gamma)+R(\gamma,\alpha,\beta)+R(\beta,\gamma,\alpha) = 0,\label{eq:firstBianchi}\\
    &\paraa{\nabla_\alpha R}(\beta,\gamma,\delta)+\paraa{\nabla_\beta R}(\gamma,\alpha,\delta)+
    \paraa{\nabla_\gamma R}(\alpha,\beta,\delta)=0\label{eq:secondBianchi},
  \end{align}
  for all $\alpha,\beta,\gamma,\delta\in\g$.
\end{proposition}

\begin{proof}
  The first Bianchi identity (\ref{eq:firstBianchi}) is proven by
  acting with $\nabla_\gamma$ on the torsion free condition
  $\nabla_\alpha \beta-\nabla_\beta\alpha-[\alpha,\beta]=0$, and then summing over cyclic
  permutations of $\alpha,\beta,\gamma$. Since $[[\alpha,\beta],\gamma]+[[\beta,\gamma],\alpha]+[[\gamma,\alpha],\beta]=0$,
  the desired result follows.  The second identity is obtained by a
  cyclic permutation (of $\alpha,\beta,\gamma$) in
  $R\paraa{\nabla_\alpha \beta-\nabla_\beta\alpha-[\alpha,\beta],\gamma,\delta}=0$. One has
  \begin{align*}
    0&=R\paraa{\nabla_\alpha \beta-\nabla_\beta\alpha-[\alpha,\beta],\gamma,\delta}+\text{cycl.}\\
    &=R(\nabla_\gamma \alpha,\beta,\delta)+R(\alpha,\nabla_\gamma \beta,\delta)-R([\alpha,\beta],\gamma,\delta) +\text{cycl.}
  \end{align*}
  On the other hand, one has
  \begin{align*}
    (\nabla_\gamma R)(\alpha,\beta,\delta) &= \nabla_\gamma R(\alpha,\beta,\delta) -R(\nabla_\gamma \alpha,\beta,\delta)\\
    &\qquad-R(\alpha,\nabla_\gamma \beta,\delta)
    -R(\alpha,\beta,\nabla_\gamma \delta),
  \end{align*}
  and substituting this into the previous equation yields
  \begin{align*}
    0 = \nabla_\gamma R(\alpha,\beta,\delta)-\paraa{\nabla_\gamma
      R}(\alpha,\beta,\delta)-R(\alpha,\beta,\nabla_\gamma \delta)
    -R([\alpha,\beta],\gamma,\delta)+\text{cycl.}
  \end{align*}
  After inserting the definition of $R$, and using that
  $[[\alpha,\beta],\gamma]+\text{cycl.}=0$, the second Bianchi identity follows.
\end{proof}

\noindent
Finally, one is able to derive the classical symmetries of the curvature tensor.

\begin{proposition}
  Let $\nabla$ be the Levi-Civita connection of a metric Lie-Rinehart
  algbera $(\A,\g,g)$ and let $R$ denote corresponding curvature. Then
  it holds that
  \begin{align}
    &R(\alpha,\beta,\gamma,\delta) = -R(\beta,\alpha,\gamma,\delta) = -R(\alpha,\beta,\delta,\gamma).\label{eq:Rsymmetry1}\\
    &R(\alpha,\beta,\gamma,\delta) = R(\delta,\gamma,\alpha,\beta),\label{eq:Rsymmetry2}
  \end{align}
  for all $\alpha,\beta,\gamma,\delta\in\g$.
\end{proposition}

\begin{proof}
  The identity $R(\alpha,\beta,\gamma,\delta)=-R(\alpha,\beta,\delta,\gamma)$ follows immediately from the
  definition of $R$. Let us now prove that
  $R(\alpha,\beta,\gamma,\delta)=-R(\beta,\alpha,\gamma,\delta)$. Starting from $\gamma(\delta(a))-\delta(\gamma(a))-[\gamma,\delta](a)=0$
  and letting $a=g(\alpha,\beta)$ yields
  \begin{align*}
    &\gamma\bracketb{g(\nabla_\delta\alpha,\beta)+g(\alpha,\nabla_\delta\beta)}
    -\delta\bracketb{g(\nabla_\gamma\alpha,\beta)+g(\alpha,\nabla_\gamma\beta)}\\
    &\qquad
    -(\nabla_{[\gamma,\delta]}\alpha,\beta)-(\alpha,\nabla_{[\gamma,\delta]}\beta)=0.
  \end{align*}
  when using that $\nabla$ is a metric connection; i.e
  $\tau(g(\alpha,\beta))=g(\nabla_\tau\alpha,\beta)+g(\alpha,\nabla_\tau\beta)$ for $\tau=\gamma,\delta,[\gamma,\delta]$.
  A further expansion using the metric property gives
  \begin{align*}
    &g(\nabla_\gamma\nabla_\delta\alpha,\beta)+g(\alpha,\nabla_\gamma\nabla_\delta\beta)
    -g(\nabla_\delta\nabla_\gamma \alpha,\beta)-g(\alpha,\nabla_\delta\nabla_\gamma\beta)\\
    &\qquad
    -g(\nabla_{[\gamma,\delta]}\alpha,\beta)-g(\alpha,\nabla_{[\gamma,\delta]}\beta)=0,
  \end{align*}
  which is equivalent to
  \begin{align*}
    g(R(\gamma,\delta)\alpha,\beta) = -g(R(\gamma,\delta)\beta,\alpha).
  \end{align*}
  Next, one can make use of equation (\ref{eq:firstBianchi}) in
  Proposition \ref{prop:bianchiIdentities}, from which it follows that
  \begin{align}
    R(\alpha,\beta,\gamma,\delta)+R(\alpha,\delta,\beta,\gamma)+R(\alpha,\gamma,\delta,\beta) = 0.\label{eq:RquadPermSymmetry}
  \end{align}
  It is a standard algebraic result that
  any quadri-linear map satisfying (\ref{eq:Rsymmetry1}) and
  (\ref{eq:RquadPermSymmetry}) also satisfies (\ref{eq:Rsymmetry2})
  (see e.g. \cite{h:diffsymspaces}).
\end{proof}

\section{\KP\ algebras}\label{sec:kpalgebras}

\noindent In this section, we shall introduce a type of Poisson
algebras, that resembles the smooth functions on an (isometrically)
embedded almost K\"ahler manifold, in such a way that an analogue of
Riemannian geometry may be developed. Namely, let us consider a unital
Poisson algebra $(\A,\{\cdot,\cdot\})$ and let $\{x^1,\ldots,x^m\}$ be
a set of distinguished elements of $\A$, corresponding to functions
providing an embedding into $\reals^m$, in the geometrical case. One
may also consider the setting of algebraic (Poisson) varieties where
$\A$ is a finitely generated Poisson algebra and $\{x^1,\ldots,x^m\}$
denotes a set of generators. Our aim is to introduce equation
\eqref{eq:kpcond.f.h} in $\A$ and investigate just how far one may
take the analogy with Riemannian geometry.  After introducing
\emph{\KP\ algebras} below, we will show that they are, in a natural
way, metric Lie-Rinehart algebras, which implies that the results of
Section~\ref{sec:metric.Lie.Rinehart} can be applied; in particular,
there exists a unique torsion-free metric connection on every \KP\
algebra. Note that Lie-Rinehart algebras related to Poisson algebras
have extensively been studied by Huebschmann (see
e.g. \cite{h:poisson.cohomology,h:exentsions.lie.rinehart}).

In Section~\ref{sec:poisson.algebraic.geometry} it was shown that the
following identity holds on an almost K\"ahler manifold:
\begin{align}
  \gamma^2\{f,h\} = -\{f,x^i\}\gb_{ij}\{x^j,x^k\}\gb_{kl}\{x^l,h\}\tag{\ref{eq:kpcond.f.h}}.
\end{align}
This equation is well-defined in a Poisson algebra, and we shall use
it to define the main object of our investigation.

\begin{definition}
  Let $\A$ be a Poisson algebra over $\Kfield$ and let
  $\{x^1,\ldots,x^m\}\subseteq\A$. Given $g_{ij}\in\A$, for
  $i,j=1,\ldots,m$, such that $g_{ij}=g_{ji}$, we say that the triple
  $\K=\paraa{\A,\{x^1,\ldots,x^m\},g}$ is a \emph{\KP-algebra} if
  there exists $\eta\in\A$ such that
  \begin{align}\label{eq:KPalgDefRel}
    \sum_{i,j,k,l=1}^m\!\!\eta\{a,x^i\}g_{ij}\{x^j,x^k\}g_{kl}\{x^l,b\}=-\{a,b\}
  \end{align}
  for all $a,b\in\A$.
\end{definition}

\begin{remark}
  From now on, we shall use the differential geometric convention that repeated indices
  are summed over from $1$ to $m$, and omit explicit summation symbols.
\end{remark}

\noindent
Given a \KP-algebra $\K$, we let $\g$ denote the $\A$-module generated
by all inner derivations, i.e.
\begin{align*}
  \g = \{a_1\{c^1,\cdot\}+\cdots+a_N\{c^N,\cdot\}:a_i,c^i\in\A\text{
    and }N\in\naturals\}.
\end{align*}
It is a standard fact that $\g$ is a Lie algebra over $\Kfield$
with respect to
\begin{align*}
  [\alpha,\beta](a) = \alpha\paraa{\beta(a)}-\beta\paraa{\alpha(a)}.
\end{align*}
The matrix $g$ induces a bilinear symmetric form on $\g$, defined by
\begin{align}\label{eq:inducedMetricDef}
  g(\alpha,\beta)=\alpha(x^i)g_{ij}\beta(x^j),
\end{align}
and we refer to $g$ as the metric on $\g$. To the metric $g$ one may
associate the map $\gh:\g\to\g^\ast$ defined as
\begin{align*}
  \gh(\alpha)(\beta) = g(\alpha,\beta).
\end{align*}

\begin{proposition}\label{prop:kp.metric.nondegenerate}
  If $\K=\paraa{\A,\{x^1,\ldots,x^m\},g}$ is a \KP-algebra then the
  metric $g$ is non-degenerate; i.e. the map $\gh:\g\to\g^\ast$ is a
  module isomorphism.
\end{proposition}

\begin{proof}
  Let us first show that $g$ is injective; i.e. we will show that
  $\gh(\alpha)(\beta)=0$, for all $\beta\in\g$, implies that
  $\alpha=0$. Thus, write $\alpha=\alpha_i\{x^i,\cdot\}$, and assume that
  $g(\alpha,\beta)=0$ for all $\beta\in\g$. In particular, we can
  choose $\beta=\eta \{c,x^k\}g_{km}\{\cdot,x^m\}$, for arbitrary
  $c\in\A$, which implies that
  \begin{align*}
    0 &= g(\alpha,\beta) = \eta \alpha_k\{x^k,x^i\}g_{ij}\{c,x^k\}g_{km}\{x^j,x^m\}\\
    &=-\alpha_k\eta\{x^k,x^i\}g_{ij}\{x^j,x^m\}g_{mk}\{x^k,c\}.
  \end{align*}
  Using the relation \eqref{eq:KPalgDefRel}, one obtains
  \begin{align*}
    \alpha_k\{x^k,c\} = 0
  \end{align*}
  for all $c\in\A$, which is equivalent to $\alpha=0$. This shows that
  $\gh$ is injective. Let us now show that $\gh$ is surjective.  Thus,
  let $\omega\in\g^\ast$ and set
  \begin{align*}
    \alpha = \eta\omega(\{x^i,\cdot\})g_{ij}\{x^j,\cdot\}\in\g,
  \end{align*}
  which gives 
  \begin{align*}
    \gh(\alpha)(a_k\{b^k,\cdot\}) &=
    \eta\omega(\{x^i,\cdot\})g_{ij}\{x^j,x^l\}g_{lm}a_k\{b^k,x^m\}\\
    &=-\eta a_k\{b^k,x^m\}g_{ml}\{x^l,x^j\}g_{ji}\omega(\{x^i,\cdot\}).
  \end{align*}
  Since $\omega$ is a module homomorphism one obtains
  \begin{align*}
    \gh(\alpha)(a_k\{b^k,\cdot\}) &=
    \omega(-\eta a_k\{b^k,x^m\}g_{ml}\{x^l,x^j\}g_{ji}\{x^i,\cdot\})\\
    &=\omega(a_k\{b^k,\cdot\}),
  \end{align*}
  by using \eqref{eq:KPalgDefRel}, which proves that every element of $\g^\ast$ is in the image
  of $\gh$. We conclude that $\gh$ is a module isomorphism.
\end{proof}

\begin{corollary}
  If $(\A,\{x^1,\ldots,x^m\},g)$ is a \KP\ algebra then $(\A,\g,g)$ is
  a metric Lie-Rinehart algebra.
\end{corollary}

\begin{proof}
  It is easy to check that $(\A,\g)$ satisfies the conditions of a
  Lie-Rinehart algebra, and
  Proposition~\ref{prop:kp.metric.nondegenerate} implies that the
  metric is non-degenerate. Hence, $(\A,\g,g)$ is a metric
  Lie-Rinehart algebra.
\end{proof}

\noindent
Let us now introduce some notation for \KP\ algebras. Thus, we set
\begin{align*}
  \P^{ij} = \{x^i,x^j\}\\
  \P^{i}(a) = \{x^i,a\},
\end{align*}
for $a\in\A$, as well as
\begin{align*}
  &\D^{ij} = \eta{\P^{i}}_k{\P^{jk}}
  =\eta\{x^i,x^l\}g_{lk}\{x^j,x^k\}\\
  &\D^i(a) = \eta\P^k(a){\P_k}^i
  =\eta\{x^k,a\}g_{kl}\{x^l,x^i\},
\end{align*}
and note that $\D^{ij}=\D^{ji}$.
With respect to this notation,  \eqref{eq:KPalgDefRel} can be stated as
\begin{align}
  \D^i(a)\P_i(b) = \{a,b\}.
\end{align}
The metric will be used to lower indices in analogy with differential
geometry. E.g.
\begin{align*}
  {\P^{i}}_j = \P^{ik}g_{kj}\qquad {\D^i}_j = \D^{ik}g_{kj}.
\end{align*}
Furthermore, one immediately derives the following useful identities 
\begin{align}\label{eq:DPDDformulas}
  {\D^{ij}}\P_j(a) = \P^i(a),\quad
  \P^{ij}\D_j(a) = \P^i(a)\text{ and }
  {\D^i}_j\D^{jk} = \D^{ik}.
\end{align}
by using \eqref{eq:KPalgDefRel}.

There is a natural embedding $\iota:\g\to\A^m$, given by
\begin{align*}
  \iota(a_i\{b^i,\cdot\}) = a_i\{b^i,x^k\}e_k,
\end{align*}
where $\{e_k\}_{k=1}^m$ denotes the canonical basis of the free module
$\A^m$. Moreover, $g$ defines a bilinear form on $\A^m$ via
\begin{align*}
  g(X,Y) = X^ig_{ij}Y^j
\end{align*}
for $X=X^ie_i\in\A^m$ and $Y=Y^ie_i\in\A^m$, and we introduce the map
$\D:\A^m\to\A^m$ by setting
\begin{align*}
  \D(X)={\D^{i}}_jX^je_i
\end{align*}
for $X=X^ie_i\in\A^m$.

\begin{proposition}\label{prop:DorthProjection}
  The map $\D:\A^m\to\A^m$ is an orthogonal projection; i.e. 
  \begin{align*}
    &\D^2(X)=\D(X)\text{ and }
    g(\D(X),Y)=g(X,\D(Y))
  \end{align*}
  for all $X,Y\in\A^m$.
\end{proposition}

\begin{proof}
  First, it is clear that $\D$ is an endomorphism of $\A^m$.  It
  follows immediately from \eqref{eq:DPDDformulas} that
  \begin{align*}
    \D^2(X) = {\D^i}_j{\D^j}_kX^ke_i = 
    {\D^i}_j{\D^{jl}}g_{lk}X^ke_i = \D^{il}g_{lk}X^ke_i = {\D^{i}}_kX^ke_i = \D(X).
  \end{align*}
  Furthermore, using that $\D^{ij} = \D^{ji}$ one finds that
  \begin{align*}
    g\paraa{\D(X),Y} &= {\D^{i}}_jX^jg_{ik}Y^k
    =X^j\D^{il}g_{lj}g_{ik}Y^k =X^jg_{lj}\D^{li}g_{ik}Y^k \\
    &=X^jg_{jl}{\D^l}_kY^k = g\paraa{X,\D(Y)},
  \end{align*}
  which completes the proof.
\end{proof}

\noindent
From Proposition~\ref{prop:DorthProjection} we conclude that
\begin{align*}
  \TA = \im(\D)
\end{align*}
is a finitely generated projective module. As a corollary, we prove
that $\g$ is a finitely generated projective module by showing that
$\g$ is isomorphic to $\TA$.

\begin{proposition}
  The map $\iota:\g\to\A^m$ is an isomorphism from $\g$ to $\TA$.
\end{proposition}

\begin{proof}
  First, it is clear from the definition that $\iota$ is a module
  homomorphism. Considered as a submodule of $\A^m$, elements of $\TA$
  can be characterized by the fact that $\D(X)=X$ for all
  $X\in\TA$. Thus, by showing that
  \begin{align*}
    \D\paraa{\iota(a_k\{b^k,\cdot\})} &= {\D^i}_ja_k\{b^k,x^j\}e_i
    =-a_k{\D^i}_j\P^j(b^k)a_ke_i\\
    &= -a_k\P^i(b^k)=\iota(a_k\{b^k,\cdot\})
  \end{align*}
  it follows that $\iota(a_k\{b^k,\cdot\})\in\TA$. Let us now show
  that $\iota$ is injective; assume that $\iota(a_k\{b^k,\cdot\})=0$,
  which implies that
  \begin{align*}
    a_k\{b^k,x^i\} = 0\quad\text{for }i=1,\ldots,m.
  \end{align*}
  Next, for arbitrary $c\in\A$, we write
  \begin{align*}
    a_k\{b^k,c\} = -\eta a_k\{b^k,x^i\}g_{ij}\P^{jl}g_{lm}\{x^m,c\},
  \end{align*}
  by using \eqref{eq:KPalgDefRel}. Since $a_k\{b^k,x^i\}=0$, one
  obtains $a_k\{b^k,c\}=0$ for all $c\in\A$.
  
  To prove that $\iota$ is surjective, we start from an arbitrary
  $X=X^ie_i\in\TA$, and note that
  \begin{align*}
    \iota\paraa{X^ig_{ij}\D^i(\cdot)} = X^ig_{ij}\D^{ik}e_k = \D(X) = X
  \end{align*}
  by using that $\D(X)=X$ for all $X\in\TA$. Hence, we may conclude
  that $\iota$ is an isomorphism from $\g$ to $\TA$.
\end{proof}

\begin{corollary}
  $\g$ is a finitely generated projective module.
\end{corollary}

\noindent
Note that the above result is clearly not dependent on whether or not
the underlying Poisson algebra has the structure of a \KP\ algebra, as
the definition of $\g$ involves only inner derivations. Hence, as soon
as the Poisson algebra admits the structure of a \KP\ algebra, it
follows that the module of inner derivations is
projective. Furthermore, the fact that $\g$ is a projective module has
several implications for the underlying Lie-Rinehart algebra
\cite{r:differential.forms,h:poisson.cohomology}. Next, let us show
that the derivations $\D^i$ generate $\g$ as an $\A$-module.

\begin{proposition}
  The $\A$-module $\g$ is generated by $\{\D^1,\ldots,\D^m\}$.
\end{proposition}

\begin{proof}
  First of all, it is clear that every element in the module generated
  by $\D^i$, written as
  \begin{align*}
    \alpha(c) = \alpha_i\D^i(c)
    = \eta\alpha_i\{x^i,x^j\}g_{jk}\{c,x^k\},
  \end{align*}
  is an element of $\g$. Conversely, let $\alpha\in\g$ be an arbitrary
  element written as
  \begin{align*}
    \alpha(c) = \sum_{N}a_N\{b^N,c\}.
  \end{align*}
  for $c\in\A$. Using the \KP\ condition \eqref{eq:KPalgDefRel} one
  may write this as
  \begin{align*}
    \alpha(a) &= \sum_N a_N\{b^N,c\}
    =-\sum_N\eta a_N\{b^N,x^i\}g_{ij}\{x^j,x^k\}g_{kl}\{x^l,c\}\\
    &=\parab{\sum_N a_N\{b^N,x^i\}g_{ij}}\D^j(c),
  \end{align*}
  which clearly lies in the module generated by $\{\D^1,\ldots,\D^m\}$.
\end{proof}

\noindent
Thus, every $\alpha\in\g$ may be written as $\alpha=\alpha_i\D^i$ for
some $\alpha_i\in\A$. It turns out that this is a very convenient way
of writing elements of $\g$, which shall be extensively used in the
following.  Note that if the \KP\ algebra comes from an almost
K\"ahler manifold $M$, then $\D^i$ is quite close to a partial
derivative on $M$ in the sense that $(\d_ax^i)g_{ik}\D^k(f)=\d_af$,
for $f\in C^\infty(M)$.

\subsection{The trace of linear maps}\label{sec:trace}

As we shall be interested in both Ricci and scalar curvature, which
are defined using traces of linear maps, we introduce
\begin{align}\label{eq:def.trace}
  \tr(L) = g\paraa{L(\D^i),\D^j}\D_{ij}.
\end{align}
for an $\A$-linear map $L:\g\to\g$.  This trace coincides with the
ordinary trace on $\g^\ast\otimes_{\A}\g$; namely, consider
\begin{align*}
  L = \sum_N\omega_N\otimes_{\A}\alpha^N\in\g^\ast\otimes_{\A}\g
\end{align*}
as a linear map $L:\g\to\g$ in the standard way via
\begin{align*}
  L(\beta) = \sum_N\omega_N(\beta)\alpha^N,
\end{align*}
together with
\begin{align*}
  \tr(L) = \sum_N\omega_N(\alpha^N).
\end{align*}
Writing $\alpha^N=\alpha^N_i\D^i$ one finds that
\begin{align*}
  g\paraa{L(\D^i),\D^j}\D_ {ij} 
  &= \sum_Ng\paraa{\omega_N(\D^i)\alpha^N_k\D^k,\D^j}\D_{ij}
    = \sum_N\omega_N(\D^i)\alpha^N_k\D^{kj}\D_{ij}\\
  &= \sum_N\omega_N(\alpha_k^N{\D^k}_i\D^i)
    =\sum_N\omega_N(\alpha_k^N\D^k)=\sum_N\omega_N(\alpha^N).
\end{align*}
In particular, this implies that the trace defined via
\eqref{eq:def.trace} is independent of the \KP\ structure.

\subsection{Morphisms of \KP\ algebras}

\noindent
As \KP\ algebras are also metric Lie-Rinehart algebras, we shall
require that a morphism of \KP\ algebras is also a morphism of metric
Lie-Rinehart algebras (as defined in
Section~\ref{sec:metric.Lie.Rinehart}). However, as the definition of
a \KP\ also involves the choice of a set of distinguished elements, we
will require a morphism to respect the subalgebra generated by these
elements. To this end, we start by making the following definition.

\begin{definition}
  Given a \KP\ algebra $(\A,\{x^1,\ldots,x^m\},g)$, let
  $\Afin\subseteq \A$ denote the subalgebra generated by
  $\{x^1,\ldots,x^m\}$.
\end{definition}

\noindent
Equipped with this definition, we introduce morphisms of \KP\ algebras in the
following way.

\begin{definition}
  Let $\K=(\A,\{x^1,\ldots,x^m\},g)$ and
  $\K'=(\A',\{y^1,\ldots,y^{m'}\},g')$ be \KP\ algebras together with
  their corresponding modules of derivations $\g$ and $\g'$,
  respectively. A \emph{morphism of \KP\ algebras} is a pair of maps
  $(\phi,\psi)$, with $\phi:\A\to\A'$ and $\psi:\g\to\g'$, such that
  $(\phi,\psi)$ is a morphism of the metric Lie-Rinehart algebras
  $(\A,\g,g)$ and $(\A,\g',g')$ and $\phi$ is a Poisson algebra
  homomorphism such that $\phi(\Afin)\subseteq \Afin'$.
\end{definition}

\noindent
Note that if the algebras are finitely generated such that $\A=\Afin$
and $\A'=\Afin'$ (which is the case in many examples), the condition
$\phi(\Afin)\subseteq\Afin'$ is automatically satisfied. Although a
morphism of \KP\ algebras is given by a choice of two maps $\phi$ and
$\psi$, it is often the case that $\phi$ determines $\psi$ in the
following sense.

\begin{proposition}\label{prop:KP.morphism.psi}
  Let
  $(\phi,\psi):(\A,\{x^1,\ldots,x^m\},g)\to(\A',\{y^1,\ldots,y^{m'}\},g')$
  be a morphism of \KP\ algebras such that for all $\alpha'\in\g'$
  \begin{align*}
    \alpha'\paraa{\phi(a)} = 0\quad\forall\,a\in\A\implies
    \alpha' = 0
  \end{align*}
  then
  \begin{align*}
    \psi\paraa{a\{b,\cdot\}_{\A}} = \phi(a)\{\phi(b),\cdot\}_{\A'}. 
  \end{align*}
\end{proposition}

\begin{proof}
  Let
  \begin{align*}
    (\phi,\psi):(\A,\{x^1,\ldots,x^m\},g)\to(\A',\{y^1,\ldots,y^{m'}\},g')
  \end{align*}
  be a morphism of \KP\ algebras fulfilling the assumption
  above. Since $\phi$ is a Poisson algebra homomorphism, one obtains
  for $\alpha=a\{b,\cdot\}_{\A}$
  \begin{align*}
    \phi\paraa{\alpha(c)} = \phi\paraa{a\{b,c\}_{\A}}
    =\phi(a)\{\phi(b),\phi(c)\}_{\A'}
  \end{align*}
  for all $a\in\A$. By the definition of a Lie-Rinehart morphism, this
  has to equal $\psi(\alpha)(\phi(c))$; i.e.
  \begin{align*}
    \psi(\alpha)(\phi(c)) = \phi(a)\{\phi(b),\phi(c)\}_{\A'}.
  \end{align*}
  Thus, $\psi(\alpha)$ agrees with
  $\phi(a)\{\phi(b),\cdot\}_{\A'}$ on the image of $\phi$, which
  implies that
  \begin{align*}
    \psi(\alpha) = \phi(a)\{\phi(b),\cdot\}_{\A'}
  \end{align*}
  since any derivation is determined by its action on the image of
  $\phi$ by assumption.
\end{proof}

\noindent
For instance, the requirements in
Proposition~\ref{prop:KP.morphism.psi} are clearly satisfied if $\phi$
is surjective.

\subsection{Construction of \KP\ algebras}\label{sec:kp.general.construction}

\noindent
Given a Poisson algebra $(\A,\{\cdot,\cdot\})$ one may ask if there
exist $\{x^1,\ldots,x^m\}$ and $g_{ij}$ such that
$(\A,\{x^1,\ldots,x^m\},g)$ is a \KP\ algebra? Let us consider the
case when $\A$ is a finitely generated algebra, and let
$\{x^1,\ldots,x^m\}$ be an arbitrary set of generators. If we denote
by $\P$ the matrix with entries $\{x^i,x^j\}$ and by $g$ the matrix
with entries $g_{ij}$, the \KP\ condition \eqref{eq:KPalgDefRel} may
be written in matrix notation as
\begin{align*}
  \eta\P g\P g\P = -\P.
\end{align*}
Given an arbitrary antisymmetric matrix $\P$, we shall find $g$ by
first writing $\P$ in a block diagonal form, with antisymmetric $2\times 2$
matrices on the diagonal. This is a well known result in linear
algebra, in which case the eigenvalues appear in the diagonal blocks. For an
antisymmetric matrix with entries in a commutative ring, a similar
result holds. 
\begin{lemma}\label{lemma:VTPV}
  Let $M_N(R)$ denote the set of $N\times N$ matrices with entries in
  $R$.  For $N\geq 2$, let $\P\in M_N(R)$ be an antisymmetric
  matric. Then there exists $V\in M_N(R)$, an antisymmetric $Q\in M_{N-2}(R)$ and
  $\lambda\in R$ such that
  \begin{center}
    $V^T\P V=\left(
      \begin{array}{c|c}
        \!\!\!\begin{array}{cc}
                0 & \lambda\\
                -\lambda & 0
              \end{array}& 0\\ \hline
        0 & Q
      \end{array}\right)$.
  \end{center}
\end{lemma}

\begin{proof}
  We shall construct the matrix $V$ by using elementary row and column
  operations. Note that if a matrix $E$ represents an elementary row
  operation, then $E^T\P E$ is obtained by applying the elementary
  operation to both the row and the corresponding column. Denoting the
  matrix elements of $\P$ by $p_{ij}$, we start by constructing a
  matrix $V_k$ such that $(V_k^T\P V_k)_{k1}=(V_k^T\P V_k)_{k2}=0$
  (which necessarily implies that also the $(1k)$ and $(2k)$ matrix
  elements are zero). To this end, let $V_k^1$ denote the matrix
  representing the elementary row operation that multiplies the $k$'th
  row by $p_{12}$, and let $V_k^2$ represent the operation that adds
  the first row, multiplied by $-p_{k2}$, to the $k$'th
  row. Furthermore, $V_k^3$ represents the operation of adding the
  second row, multiplied by $p_{k1}$, to the $k$'th row. Setting
  $V_k=V_k^1V_k^2V_k^3$ it is easy to see that $V_k^T\P V_k$ is an
  antisymmetric matrix where the $(1k)$, $(2k)$, $(k1)$ and $(k2)$
  matrix elements are zero. Consequently, we set $V=V_3V_4\cdots V_N$
  and conclude that $V^T\P V$ is of the desired form.
\end{proof}

\begin{proposition}\label{prop:VTPV}
  Let $\P\in M_N(R)$ be an antisymmetric matric, and let $\hat{N}$ denote
  the integer part of $N/2$. Then there exists $V\in M_N(R)$ and
  $\lambda_1,\ldots,\lambda_{\hat{N}}\in R$ such that
  \begin{align*}
    &V^T\P V = \diag(\Lambda_1,\ldots,\Lambda_{\hat{N}})\quad\text{if $N$ is even,}\\
    &V^T\P V = \diag(\Lambda_1,\ldots,\Lambda_{\hat{N}},0)\quad\text{if $N$ is odd,}
  \end{align*}
  where
  \begin{align*}
    \Lambda_k =
    \begin{pmatrix}
      0 &\lambda_k \\
      -\lambda_k & 0
    \end{pmatrix}.
  \end{align*}
\end{proposition}

\begin{proof}
  Let us prove the statement by using induction together with
  Lemma~\ref{lemma:VTPV}. Thus, assume that there exists
  $V\in M_N(R)$ such that
  \begin{align*}
    V^T\P V = \diag(\Lambda_1,\ldots,\Lambda_k,Q_{k+1})
  \end{align*}
  where $Q_{k+1}\in M_{N-2k}$ is an antisymmetric matrix. Clearly, by
  Lemma~\ref{lemma:VTPV}, this holds true for $k=1$. Next, assume that
  $N-2k\geq 2$. Applying Lemma~\ref{lemma:VTPV} to $Q_{k+1}$ we
  conclude that there exists $V_{k+1}\in M_{N-2k}(R)$ such that
  $V_{k+1}^TQ_{k+1}
  V_{k+1}=\diag(\Lambda_{k+1},Q_{k+2})$. Furthermore, defining
  $W_{k+1}\in M_{N}(R)$ by $W_{k+1}=\diag(\mid_{2k},V_{k+1})$ one finds
  that
  \begin{align*}
    (VW_{k+1})^T\P (VW_{k+1})=\diag(\Lambda_1,\ldots,\Lambda_{k+1},Q_{k+2}).
  \end{align*}
  By induction, it follows that one may repeat this procedure until
  $N-2k<2$. If $N$ is even, then $N-2k=0$ and the statement
  follows. If $N$ is odd, then $N-2k=1$ and, since $V^T\P V$ is
  antisymmetric, it follows that the $(NN)$ matrix element is zero,
  giving the stated result.
\end{proof}

\noindent
Returning to the case of a Poisson algebra generated by
$x^1,\ldots,x^m$, assume for the moment that $m=2N$ for a positive
integer $N$. By Proposition~\ref{prop:VTPV}, there exists a matrix $V$
\begin{align*}
  V^T\P V = \P_0
\end{align*}
where $\P_0$ is a block diagonal matrix of the form
\begin{align*}
  \P_0 = \diag(\Lambda_1,\ldots,\Lambda_N)
\end{align*}
with
\begin{align*}
  \Lambda_k=
  \begin{pmatrix}
    0 & \lambda_k\\
    -\lambda_k & 0
  \end{pmatrix}.
\end{align*}
In the same way, defining $g_0 = \diag(g_1,\ldots,g_N)$ with
\begin{align*}
  &g_k = \frac{\lambda}{\lambda_k}
  \begin{pmatrix}
    1 & 0\\
    0 & 1
  \end{pmatrix}\\
  &\lambda=\lambda_1\cdots\lambda_N
\end{align*}
we set $g = Vg_0V^T$. Noting that
\begin{align*}
  \P_0g_0\P_0g_0\P_0=-\lambda^2\P_0
\end{align*}
one finds
\begin{align*}
  0 &= \P_0g_0\P_0g_0\P_0+\lambda^2\P_0
  =V^T\P V g_0 V^T\P_0 VgV^T\P_0 V+\lambda^2V^T\P V\\
  &= V^T\paraa{\P g\P g\P+\lambda^2\P}V
\end{align*}
It is a general fact that for an arbitrary matrix $V$ there exists a
matrix $\Vt$ such that $\Vt V=V\Vt = (\det V)\mid$. Multiplying the
above equation from the left by $\Vt^T$ and from the right by $\Vt$
yields
\begin{align}
  \det(V)^2\paraa{\P g\P g\P+\lambda^2\P} = 0.
\end{align}
As long as $\det(V)$ is not a zero divisor, this implies that
\begin{align*}
  \P g\P g\P=-\lambda^2\P.
\end{align*}
Thus, given a finitely generated Poisson algebra $\A$, the above
procedure gives a rather general way to associate a localization
$\A[\lambda^{-1}]$ and a metric $g$ to $\A$, such that
$(\A[\lambda^{-1}],\{x^1,\ldots,x^m\},g)$ is a \KP\ algebra.  Note
that the above argument, with only slight notational changes, also
applies to the case when $m$ is odd, in which case an extra block of
$0$ will appear in $\P_0$.

\section{The Levi-Civita connection}\label{sec:levi.civita}

\noindent
Since every \KP\ algebra is also a metric Lie-Rinehart algebra, the
results of Section \ref{sec:metric.Lie.Rinehart} immediately
applies. In particular, there exists a unique torsion-free and metric
connection on the module $\g$. In this section, we shall derive an
explicit expression for the Levi-Civita connection of an arbitrary
\KP\ algebra. It turns out to be convenient to formulate the results
in terms of the generators $\{\D^1,\ldots,\D^m\}$.  Kozul's formula
gives the connection as
\begin{equation}
  \begin{split}
    2g&(\nabla_{\D^i}\D^j,\D^k)=
    \D^i\paraa{g(\D^j,\D^k)}+\D^j\paraa{g(\D^k,\D^i)}-\D^k\paraa{g(\D^i,\D^j)}\\
    &-g([\D^j,\D^k],\D^i)+g([\D^k,\D^i],\D^j)+g([\D^i,\D^j],\D^k),    
  \end{split}\label{eq:Di.Kozuls.formula}
\end{equation}
and one notes that an element $\alpha=a\{b,\cdot\}\in\g$ may be
recovered from $g(\alpha,\D^i)$ as
\begin{align*}
  g(\alpha,\D^i)\D_i(f) = a\{b,x^k\}{\D^i}_k\D_i(f)
  =a\{b,x^k\}\D_k(f)=a\{b,f\} = \alpha(f).
\end{align*}
Thus, one immediately obtains
$\nabla_{\D^i}\D^j=g(\nabla_{\D^i}\D^j,\D^k)\D_k$. However, it turns
out that one can obtain a more compact formula for the connection. Let
us start by proving the following result.

\begin{lemma}\label{lemma:DiDj.commutator}
  $g([\D^i,\D^j],\D^k)=\D^i\paraa{\D^{jk}}-\D^j\paraa{\D^{ik}}$.
\end{lemma}

\begin{proof}
  For convenience, let us introduce the notation $\Ph^{ij}=\eta\{x^i,x^j\}$ and,
  consequently, $\Phind{i}{j}=\Ph^{ik}g_{kj}$. In this notation, one
  finds $\D^i(a)=\Phind{i}{j}\{a,x^j\}$. Thus, one obtains
  \begin{align*}
    g([\D^i,\D^j],\D^k)
    &=[\D^i,\D^j](x^l){\D^{k}}_l
      =\Phind{i}{m}\{\D^{jl},x^m\}{\D^k}_l-\Phind{j}{n}\{\D^{il},x^n\}{\D^k}_l\\
    &=\paraa{\Phind{i}{m}\{\Phind{j}{n}\{x^l,x^n\},x^m\}-\Phind{j}{n}\{\Phind{i}{m}\{x^l,x^m\},x^n\}}{\D^k}_l\\
    &=\Phind{i}{m}\Phind{j}{n}\paraa{-\{\{x^n,x^l\},x^m\}-\{\{x^l,x^m\},x^n\}}{\D^k}_l\\
    &\qquad+\paraa{\Phind{i}{m}\{\Phind{j}{n},x^m\}\{x^l,x^n\}
      -\Phind{j}{n}\{\Phind{i}{m},x^n\}\{x^l,x^m\}}{\D^k}_l\\
    &=\Phind{i}{m}\Phind{j}{n}\{\{x^m,x^n\},x^k\}
      +\Phind{i}{m}\{\Phind{j}{n},x^m\}\{x^k,x^n\}\\
      &\qquad-\Phind{j}{n}\{\Phind{i}{m},x^n\}\{x^k,x^m\},
  \end{align*}
  by using the Jacobi identity together with
  $\{a,x^i\}{\D^k}_i = \{a,x^k\}$. Furthermore, in the second and
  third term, one uses Leibniz's rule to obtain
  \begin{align*}
    g([&\D^i,\D^j],\D^k)
    = \Phind{i}{m}\Phind{j}{n}\{\{x^m,x^n\},x^k\}
      +\Phind{i}{m}\{\Phind{j}{n}\{x^k,x^n\},x^m\}\\
           &-\Phind{i}{m}\Phind{j}{n}\{\{x^k,x^n\},x^m\}
             -\Phind{j}{n}\{\Phind{i}{m}\{x^k,x^m\},x^n\}
             +\Phind{j}{n}\Phind{i}{m}\{\{x^k,x^m\},x^n\}\\
           &=\Phind{i}{m}\Phind{j}{n}\paraa{\{\{x^m,x^n\},x^k\}
             +\{\{x^n,x^k\},x^m\}+\{\{x^k,x^m\},x^n\}}\\
           &\qquad+\D^i\paraa{\D^{jk}}-\D^j(\D^{ik})
             =\D^i\paraa{\D^{jk}}-\D^j(\D^{ik}),
  \end{align*}
  by again using the Jacobi identity.
\end{proof}

\noindent
The above result allows for the following formulation of the
Levi-Civita connection for a \KP\ algebra.

\begin{proposition}\label{prop:levi.civita.connection}
  If $\nabla$ denotes the Levi-Civita connection of a \KP\ algebra $\K$ then
  \begin{align}
    \nabla_{\D^i}\D^j = \half\D^i(\D^{jk})\D_k-\half\D^j(\D^{ik})\D_k+\half\D^k(\D^{ij})\D_k,\label{eq:levi.civita.connection.formula}
  \end{align}
  or, equivalently, $\nabla_{\D^i}\D^j={\Gamma^{ij}}_k\D^k$ where
  \begin{align}
    {\Gamma^{ij}}_k= \half\D^i(\D^{jl})\D_{lk}-\half\D^j(\D^{il})\D_{lk}+\half\D_k(\D^{ij}).\label{eq:christoffel.symbols}
  \end{align}
\end{proposition}

\begin{proof}
  Since $g(\D^i,\D^j)=\D^{ij}$, Kozul's formula
  \eqref{eq:Di.Kozuls.formula} together with
  Lemma~\ref{lemma:DiDj.commutator} gives
  \begin{align*}
    2g(\nabla_{\D^i}\D^j,\D^k)&=
                                \D^i(\D^{jk})+\D^j(\D^{ki})-\D^k(\D^{ij})
                                -\D^j(\D^{ki})+\D^k(\D^{ji})\\
                              &+\D^k(\D^{ij})-\D^i(\D^{kj})+\D^i(\D^{jk})-\D^j(\D^{ik})\\
                              &= \D^i(\D^{jk})-\D^j(\D^{ki})+\D^k(\D^{ij}),
  \end{align*}
  which proves \eqref{eq:levi.civita.connection.formula}. The fact
  that one may write the connection as
  $\nabla_{\D^i}\D^j = {\Gamma^{ij}}_k\D^k$ follows from
  $\D_{ij}\D^j=\D_i$ and $\D^k(a)\D_k(b)=\D_k(a)\D^k(b)$.
\end{proof}

\noindent
Thus, for arbitrary elements of $\g$, one obtains
\begin{align}
  \nabla_{\alpha}\beta = \alpha(\beta_i)\D^i + {\Gamma^{ij}}_k\alpha_i\beta_j\D^k 
\end{align}
where $\alpha=\alpha_i\D^i$ and $\beta=\beta_i\D^i$, and curvature is
readily introduced as
\begin{align*}
  R(\alpha,\beta)\gamma = \nabla_\alpha\nabla_\beta\gamma-\nabla_\beta\nabla_\alpha\gamma-\nabla_{[\alpha,\beta]}\gamma.
\end{align*}
Ricci curvature is defined as
\begin{align*}
  &\Ric(\alpha,\beta) = \tr\paraa{\gamma\to R(\gamma,\alpha)\beta}
\end{align*}
and using the trace from Section~\ref{sec:trace}, one obtains
\begin{align*}
  &\Ric(\alpha,\beta) = g(R(\D^i,\alpha)\beta,\D^j)\D_{ij}.
\end{align*}
To define the scalar curvature, one considers the Ricci curvature as a linear map
\begin{align*}
  \Ric: \g \to\g\quad\text{with}\quad
  \Ric(\alpha) = \Ric(\alpha,\D^i)\D_i,
\end{align*}
giving
\begin{align*}
  S = \tr\paraa{\alpha\to\Ric(\alpha)}
  = g\paraa{R(\D^i,\D^k)\D^l,\D^j}\D_{ij}\D_{kl}.
\end{align*}

\noindent
Note that since the metric is nondegenerate, there exists a unique
element $\nabla f\in\g$ such that $g(\nabla f,\alpha) = \alpha(f)$ for
all $\alpha\in\g$; we call $\nabla f$ the \emph{gradient of $f$}. Now,
it is easy to see that
\begin{align*}
  \nabla f = \D_i(f)\D^i
\end{align*}
since
\begin{align*}
  g(\D_i(f)\D^i,\alpha_j\D^j) = \D_i(f)\alpha_j\D^{ij}  = \alpha_j\D^j(f)
  =\alpha(f).
\end{align*}
The divergence of an element $\alpha\in\g$ is defined as
\begin{align*}
  \div(\alpha) = \tr(\beta\to\nabla_{\beta}\alpha),
\end{align*}
and, finally, the Laplacian
\begin{align*}
  \Delta(f)=\div(\nabla f).
\end{align*}

\section{Examples}
\label{sec:examples}

\noindent
As shown in Section~\ref{sec:poisson.algebraic.geometry}, the algebra
of smooth functions on an almost K\"ahler manifold $M$ becomes a \KP\
algebra when choosing $x^1,\ldots,x^m$ to be embedding coordinates,
providing an isometric embedding into $\reals^m$, endowed with the
standard Euclidean metric. (Recall that, by Nash's
theorem~\cite{nash:imbedding}, such an embedding always exists.)  In
this section, we shall present examples of a more algebraic nature to
illustrate the fact that algebras of smooth functions are not the only
examples of \KP\ algebras.

Keeping in mind the general construction procedure in
Section~\ref{sec:kp.general.construction}, we consider finitely
generated Poisson algebras with a low number of generators. 

\subsection{Poisson algebras generated by two elements}
\label{sec:ex.two.elements}

\noindent Let $\A$ be a unital Poisson algebra generated by the two
elements $x^1=x\in\A$ and $x^2=y\in\A$, and set
\begin{align*}
  \P =
  \begin{pmatrix}
    0 & \{x,y\}\\
    -\{x,y\} & 0
  \end{pmatrix}
\end{align*}
It is easy to check that for an arbitrary symmetric matrix $g$
\begin{align*}
  \P g\P g\P = -\{x,y\}^2\det(g)\P.
\end{align*}
Thus, as long as $\{x,y\}^2\det(g)$ is not a zero-divisor, one may
localize to obtain a \KP\ algebra
\begin{align*}
 \K=(\A[(\{x,y\}^2\det(g))^{-1}],\{x,y\},g).
\end{align*}
For the sake of illustrating the concepts and formulas we have
developed so far, let us explicitly work out an example based on an
algebra $\A_0$, generated by two elements. Let us start by choosing an
element $\lambda\in\A_0$ for which the localization
$\A=\A_0[p^{-1},\lambda^{-1}]$ exists, and then defining the metric as
\begin{align*}
  g = \frac{1}{\lambda}
  \begin{pmatrix}
    1 & 0 \\ 0 & 1
  \end{pmatrix}
\end{align*}
From the above considerations, we know that $(\A,\{x,y\},g)$ is a \KP\
algebra with $\eta=\lambda^2/p^2$, where $p=\{x,y\}$. For convenience
we also introduce $\gamma=p/\lambda$ such that $\eta=1/\gamma^2$. Let
us start by computing the derivations $\D^x=\D^1$ and $\D^y=\D^2$,
which generate the module $\g$:
\begin{align*}
  &\D^x = \eta\{x,x^i\}g_{ij}\{\cdot,x^j\} = \frac{\lambda}{p}\{\cdot,y\}
  =-\gover\{y,\cdot\}\\
  &\D^y = \eta\{y,x^i\}g_{ij}\{\cdot,x^j\}=-\frac{\lambda}{p}\{\cdot,x\}
    =\gover\{x,\cdot\}
\end{align*}
as well as
\begin{align*}
  &\D_x = g_{1k}\D^k = \frac{1}{\lambda}\D^x\mathand
    \D_y = g_{2k}\D^k = \frac{1}{\lambda}\D^y.
\end{align*}
Moreover, they provide an orthogonal set of generators since
\begin{align*}
  &g(\D^x,\D^x) = \gover\{y,x^i\}g_{ij}\gover\{y,x^j\} = \frac{1}{\gamma^2}\frac{p^2}{\lambda}=\lambda\\
  &g(\D^y,\D^y) = \lambda\qquad g(\D^x,\D^y) = 0,
\end{align*}
and one obtains
\begin{align*}
  (\D^{ij})=\paraa{g(\D^i,\D^j)}=
  \begin{pmatrix}
    \lambda & 0 \\
    0 & \lambda
  \end{pmatrix}.
\end{align*}
Note that $\g$ is a free module with basis $\{\D^x,\D^y\}$ since
\begin{align*}
  &a\D^x+b\D^y = 0\Rightarrow
  \begin{cases}
    a\D^x(x)+b\D^y(x)=0\\
    a\D^x(y)+b\D^y(y)=0
  \end{cases}\Rightarrow
  \begin{cases}
    -a\gover\{y,x\} = 0\\
    b\gover\{x,y\} = 0
  \end{cases}\Rightarrow
  \begin{cases}
    a = 0\\
    b = 0
  \end{cases}
\end{align*}
by using that $\lambda$ is invertible.

Let us introduce the derivation $\D^\lambda=\gi\{\lambda,\cdot\}$ and note that
\begin{align*}
  \D^\lambda = [\D^x,\D^y]=\frac{1}{\lambda}\D^x(\lambda)\D^y-\frac{1}{\lambda}\D^y(\lambda)\D^x.
\end{align*}
From Proposition~\ref{prop:levi.civita.connection} one computes the connection:
\begin{align*}
  \nabla_{\D^x}\D^x &= \half \D^1(\D^{1k})\D_k-\half\D^1(\D^{1k})\D_k+\half\D^k(\D^{11})\D_k\\
                    &=\half\D^x(\lambda)\D_x+\half\D^y(\lambda)\D_y=\half\D^i(\lambda)\D_i
\end{align*}
and similarly
\begin{align*}
  \nabla_{\D^y}\D^y&=\half\D^x(\lambda)\D_x+\half\D^y(\lambda)\D_y=\nabla_{\D^x}\D^x\\
  \nabla_{\D^x}\D^y &=\half\D^x(\lambda)\D_y-\half\D^y(\lambda)\D_x = \D^\lambda\\
  \nabla_{\D^y}\D^x &=\half\D^y(\lambda)\D_x-\half\D^x(\lambda)\D_y=-\D^\lambda
\end{align*}
Moreover, the curvature can readily be computed
\begin{align*}
  &R(\D^x,\D^y)\D^x = \bracketc{\D_x(\lambda)^2+\D_y(\lambda)^2
    -\half\D_x\paraa{\D^x(\lambda)}-\half\D_y\paraa{\D^y(\lambda)}}\D^y\\
  &R(\D^x,\D^y)\D^y =-\bracketc{\D_x(\lambda)^2+\D_y(\lambda)^2
    -\half\D_x\paraa{\D^x(\lambda)}-\half\D_y\paraa{\D^y(\lambda)}}\D^x,
\end{align*}
as well as the scalar curvature
\begin{align*}
  S = \frac{1}{\lambda}\parab{
  \D_x\paraa{\D^x(\lambda)}+\D_y\paraa{\D^y(\lambda)}-2\D_x(\lambda)^2-2\D_y(\lambda)^2}.
\end{align*}
Moreover, one finds that
\begin{align*}
  \nabla f &= \D^x(f)\D_x+\D^y(f)\D_y\\
  \div(\alpha_x&\D^x+\alpha_y\D^y) = \D^x(\alpha_x)+\D^y(\alpha_y)\\
  \Delta(f) &= \D^x\paraa{\D_x(f)}+\D^y\paraa{\D_y(f)}\\
  &=\D_x\paraa{\D^x(f)}+\D_y\paraa{\D^y(f)}-\D_x(\lambda)\D_x(f)-\D_y(\lambda)\D_y(f).
\end{align*}

\subsection{Poisson algebras generated by three elements}
\label{sec:ex.three.elements}

\noindent Let $\A$ be a unital Poisson algebra generated by
$x^1=x,x^2=y,x^3=z\in\A$. Writing $\{x,y\}=a$, $\{y,z\}=b$ and $\{z,x\}=c$, i.e.
\begin{align*}
  \P = 
  \begin{pmatrix}
    0  & a  & -c \\
    -a & 0  & b\\
    c  & -b & 0
  \end{pmatrix},
\end{align*}
one readily checks that for an arbitrary symmetric matrix $g$
\begin{align*}
  \P g\P g\P = -\tau\P
\end{align*}
with
\begin{align*}
  \tau = a^2|g|_{33}+b^2|g|_{11}+c^2|g|_{22}+2ab|g|_{31}-2ac|g|_{32} -2bc|g|_{21},
\end{align*}
where $|g|_{ij}$ denotes the determinant of the matrix obtained from
$g$ by deleting the $i$'th row and the $j$'th column. Thus, one may construct the \KP\ algebra
\begin{align*}
  \K = \{\A[\tau^{-1}],\{x,y,z\},g\}.
\end{align*}
In particular, if $g=\diag(\lambda,\lambda,\lambda)$, then $\tau=\lambda^2(a^2+b^2+c^2)$.

Let us now construct a particular class of algebras with a natural
geometric interpretation and a close connection to algebraic
geometry. Let $\reals[x,y,z]$ be the polynomial ring in three
variables over the real numbers, and write $x^1=x$, $x^2=y$ and
$x^3=z$. For arbitrary $C\in\reals[x,y,z]$, it is straight-forward to
show that
\begin{align*}
  \{x^i,x^j\} = \eps^{ijk}\d_kC,
\end{align*}
where $\eps^{ijk}$ denotes the totally antisymmetric symbol with
$\eps^{123}=1$, defines a Poisson structure on $\reals[x,y,z]$ which
is well-defined on the quotient $\A_C=\reals[x,y,z]/(C)$
since
\begin{align*}
  \{x^i,C\} = \{x^i,x^j\}\d_jC = \eps^{ijk}(\d_kC)(\d_jC)=0.
\end{align*}
In the spirit of algebraic geometry, the algebra $\A_C$ has a natural
interpretation as the polynomial functions on the level set
$C(x,y,z)=0$ in $\reals^3$. Choosing the metric
$g_{ij}=\delta_{ij}\mid$ (corresponding to the Euclidean metric on
$\reals^3$) one obtains a \KP\ algebra
$(\widehat{\A}_C,\{x,y,z\},g)$ where
\begin{align*}
  \widehat{\A}_C=\A_C[\tau^{-1}]\mathand
  \tau = \paraa{\d_xC}^2+\paraa{\d_yC}^2+\paraa{\d_zC}^2,
\end{align*}
with $\eta=\tau^{-1}$. Note that the points in $\reals^3$, for which
$\tau(x,y,z)=0$, coincide with the singular points of $C(x,y,z)=0$;
i.e. points where $\d_xC=\d_yC=\d_zC=0$.

As an illustration, let us choose
$C=\tfrac{1}{2}(ax^2+by^2+cz^2-\mid)$ for $a,b,c\in\reals$, giving
\begin{align*}
  \{x,y\} = cz,\quad\{y,z\}=ax\mathand\{z,x\}=by.
\end{align*}
and
\begin{align*}
  \eta = \paraa{a^2x^2+b^2y^2+c^2z^2}^{-1}
\end{align*}
together with
\begin{align*}
  (\D^{ij})= \eta
  \begin{pmatrix}
    b^2y^2+c^2z^2 & -abxy & -acxz\\
    -abxy & a^2x^2+c^2z^2 & -bcyz \\
    -acxz & -bcyz & a^2x^2+b^2y^2
  \end{pmatrix}.
\end{align*}
A straight-forward, but somewhat lengthy, calculation gives
\begin{align*}
  &R(\D^x,\D^y)
  \begin{pmatrix}
    \D^x \\ \D^y \\ \D^z
  \end{pmatrix}
  =cz\,\Rh
  \begin{pmatrix}
    \D^x \\ \D^y \\ \D^z
  \end{pmatrix}\qquad
  R(\D^y,\D^z)
  \begin{pmatrix}
    \D^x \\ \D^y \\ \D^z
  \end{pmatrix}
  = ax\,\Rh
  \begin{pmatrix}
    \D^x \\ \D^y \\ \D^z
  \end{pmatrix}\\
  &  R(\D^z,\D^x)
  \begin{pmatrix}
    \D^x \\ \D^y \\ \D^z
  \end{pmatrix}
  =by\,\Rh
  \begin{pmatrix}
    \D^x \\ \D^y \\ \D^z
  \end{pmatrix}\quad\text{where}\quad
  \Rh =abc\eta^3
  \begin{pmatrix}
    0 & -cz & by \\
    cz & 0 & -ax \\
    -by & ax & 0 
  \end{pmatrix},  
\end{align*}
and the scalar curvature becomes
\begin{align*}
  S = 2abc\eta^2.
\end{align*}

\section{Summary}\label{sec:summary}

\noindent
In this note, we have introduced the concept of \KP\ algebras as a
mean to study Poisson algebras from a metric point of view. As shown,
the single relation \eqref{eq:KPalgDefRel} has consequences that allow
for an identification of geometric objects in the algebra, which share
crucial properties with their classical counterparts.  The idea behind
the construction was to identify a distinguished set of elements in
the algebra that serve as ``embedding coordinates'', and then
construct the projection operator $\D$ that projects from the tangent
space of the ambient manifold onto that of the embedded
submanifold. It is somewhat surprising that \eqref{eq:KPalgDefRel}
encodes the crucial elements that are needed for the algebra to
resemble an algebra of functions on an almost K\"ahler manifold.

As outlined in Section~\ref{sec:kp.general.construction}, a large
class of Poisson algebras admit a \KP\ algebra as an associated
localization, which shows a certain generality of our treatment. Thus,
even if one is not interested in metric structures on a Poisson
algebra, the tools we have developed might be of help. For instance,
if a Poisson algebra can be given the structure of a \KP\ algebra, one
immediately concludes that the module generated by the inner
derivations is a finitely generated projective module. A statement
which is clearly independent of any metric structure. A comparison
with differential geometry is close at hand, where the structure of a
Riemannian manifold can be used to prove results about the underlying
manifold (or even the topological structure).

Let us end with a brief outlook. After having studied the basic
properties of \KP\ algebras in this paper, there are several natural
questions that can be studied. For instance, what is the interplay
between the cohomology (of Lie-Rinehart algebras) and the Levi-Civita
connection? Can one perhaps use the connection to compute cohomology?
Is there a natural way to study the moduli spaces of Poisson algebras;
i.e. how many (non-isomorphic) \KP\ structures does there exist on a
given Poisson algebra? We hope to return to these, and many other
interesting questions, in the near future.

\section*{Acknowledgments}

\noindent
We would like to thank M. Izquierdo for ideas and
discussions. Furthermore, J. A. is supported by the Swedish Research
Council.

\bibliographystyle{alpha}
\bibliography{kpalgebras}  

\begin{thebibliography}{AHH12}

\bibitem[AH14]{ah:pseudo.riemannian}
J.~Arnlind and G.~Huisken.
\newblock Pseudo-{R}iemannian geometry in terms of multi-linear brackets.
\newblock {\em Lett. Math. Phys.}, 104(12):1507--1521, 2014.

\bibitem[AHH12]{ahh:multi.linear}
J.~Arnlind, J.~Hoppe, and G.~Huisken.
\newblock Multi-linear formulation of differential geometry and matrix
  regularizations.
\newblock {\em J. Differential Geom.}, 91(1):1--39, 2012.

\bibitem[Ber79]{b:poisson.algebraic.geometry}
R.~Berger.
\newblock G\'eom\'etrie alg\'ebrique de {P}oisson.
\newblock {\em C. R. Acad. Sci. Paris S\'er. A-B}, 289(11):A583--A585, 1979.

\bibitem[Bry88]{b:differential.complex.poisson}
J.-L. Brylinski.
\newblock A differential complex for {P}oisson manifolds.
\newblock {\em J. Differential Geom.}, 28(1):93--114, 1988.

\bibitem[BS10]{bs:curvature.gravity.matrix.models}
D.~N. Blaschke and H.~Steinacker.
\newblock Curvature and gravity actions for matrix models.
\newblock {\em Classical Quantum Gravity}, 27(16):165010, 15, 2010.

\bibitem[Hel01]{h:diffsymspaces}
S.~Helgason.
\newblock {\em Differential geometry, {L}ie groups, and symmetric spaces},
  volume~34 of {\em Graduate Studies in Mathematics}.
\newblock American Mathematical Society, Providence, RI, 2001.

\bibitem[Her53]{h:pseudo.algebre.lie}
J.-C. Herz.
\newblock Pseudo-alg\`ebres de {L}ie. {I}.
\newblock {\em C. R. Acad. Sci. Paris}, 236:1935--1937, 1953.

\bibitem[Hue90]{h:poisson.cohomology}
J.~Huebschmann.
\newblock Poisson cohomology and quantization.
\newblock {\em J. Reine Angew. Math.}, 408:57--113, 1990.

\bibitem[Hue99]{h:exentsions.lie.rinehart}
J.~Huebschmann.
\newblock Extensions of {L}ie-{R}inehart algebras and the {C}hern-{W}eil
  construction.
\newblock In {\em Higher homotopy structures in topology and mathematical
  physics ({P}oughkeepsie, {NY}, 1996)}, volume 227 of {\em Contemp. Math.},
  pages 145--176. Amer. Math. Soc., Providence, RI, 1999.

\bibitem[Kar02]{k:covariant.quantization.separation}
A.~Karabegov.
\newblock A covariant {P}oisson deformation quantization with separation of
  variables up to the third order.
\newblock {\em Lett. Math. Phys.}, 61(3):255--261, 2002.

\bibitem[Kon03]{k:deformation.quantization}
M.~Kontsevich.
\newblock Deformation quantization of {P}oisson manifolds.
\newblock {\em Lett. Math. Phys.}, 66(3):157--216, 2003.

\bibitem[Koz60]{k:fibre.bundles}
J.~L. Kozul.
\newblock {\em Lectures on fibre bundles and differential geometry}.
\newblock Tata Institute of Fundamental Research, Bombay, 1960.

\bibitem[Lic77]{l:poisson.lie}
A.~Lichnerowicz.
\newblock Les vari\'et\'es de {P}oisson et leurs alg\`ebres de {L}ie
  associ\'ees.
\newblock {\em J. Differential Geometry}, 12(2):253--300, 1977.

\bibitem[Nas56]{nash:imbedding}
J.~Nash.
\newblock The imbedding problem for {R}iemannian manifolds.
\newblock {\em Ann. of Math. (2)}, 63:20--63, 1956.

\bibitem[Nel67]{n:tensoranalysis}
E.~Nelson.
\newblock {\em Tensor Analysis}.
\newblock Princeton University Press, Princeton, New Jersey, 1967.

\bibitem[Pal61]{p:cohomology.lie.rings}
R.~S. Palais.
\newblock The cohomology of {L}ie rings.
\newblock In {\em Proc. {S}ympos. {P}ure {M}ath., {V}ol. {III}}, pages
  130--137. American Mathematical Society, Providence, R.I., 1961.

\bibitem[Rin63]{r:differential.forms}
G.~S. Rinehart.
\newblock Differential forms on general commutative algebras.
\newblock {\em Trans. Amer. Math. Soc.}, 108:195--222, 1963.

\bibitem[Wei83]{a:local.structure.poisson}
A.~Weinstein.
\newblock The local structure of {P}oisson manifolds.
\newblock {\em J. Differential Geom.}, 18(3):523--557, 1983.

\end{thebibliography}

\end{document}